\documentclass[12pt]{article}

\usepackage{hyperref}
\usepackage{amsmath}
\usepackage{amssymb}
\usepackage{amsfonts}
\usepackage{amsopn}
\usepackage{amsthm}
\usepackage[all]{xy}
\usepackage{longtable}

\swapnumbers
\theoremstyle{plain}
\newtheorem{thm}{Theorem}[section]
\newtheorem{lem}[thm]{Lemma}

\theoremstyle{definition}
\newtheorem{ntt}[thm]{}
\newtheorem{cor}[thm]{Corollary}
\newtheorem{rem}[thm]{Remark}

\newcommand{\ff}{\mathbb{F}}
\newcommand{\E}{\mathrm{E}}
\newcommand{\F}{\mathrm{F}_4}
\newcommand{\G}{\mathrm{G}_2}
\newcommand{\BG}{\overline{G}}
\newcommand{\w}{\bar{\omega}}

\DeclareMathOperator{\PGL}{\mathrm{PGL}}
\DeclareMathOperator{\PGO}{\mathrm{PGO}}
\DeclareMathOperator{\PGSp}{\mathrm{PGSp}}
\DeclareMathOperator{\CH}{\mathrm{CH}}
\DeclareMathOperator{\Ch}{\mathrm{Ch}}
\DeclareMathOperator{\CHO}{\overline{\CH}}
\DeclareMathOperator{\ChO}{\overline{\Ch}}
\DeclareMathOperator{\charact}{\mathrm{char}} 

\title{ADDENDUM TO:\\
Generically split projective homogeneous varieties}
\author{Viktor Petrov, Nikita Semenov\footnote{The authors gratefully acknoweledge the hospitality and support of the
Max-Planck Institute for Mathematics, Bonn.}}
\date{}

\begin{document}
\maketitle

\begin{abstract}
In this addendum we generalize some results of our article \cite{PS10}.
More precisely, we remove all restrictions on the characteristic of the base
field (in \cite{PS10} we assumed that the characteristic is different
from any torsion prime of the group), and complete our classification
by the last missing case, namely $\PGO_{2n}^+$. We follow our notation
from \cite{PS10}.
\end{abstract}

\section{Chow rings of reductive groups}\label{sec2}

\begin{ntt}
Let $G_0$ be a split reductive algebraic group defined over a field $k$.
We fix a split maximal torus $T$ in $G_0$ and a Borel subgroup $B$ of $G_0$
containing $T$ and defined over $k$. 
We denote by $\Phi$ the root system of $G_0$,
by $\Pi$ the set of simple roots of $\Phi$
with respect to $B$, and by $\widehat T$ the group of characters of $T$.
Enumeration of simple roots follows Bourbaki.

Any projective $G_0$-homogeneous variety $X$
is isomorphic to $G_0/P_\Theta$, where $P_\Theta$ stands for the (standard) parabolic
subgroup corresponding to a subset $\Theta\subset\Pi$. As $P_i$ we denote the
maximal parabolic subgroup $P_{\Pi\setminus\{\alpha_i\}}$ {\it of type $i$}.
\end{ntt}

Consider the characteristic map $c\colon S(\widehat T)\to\CH^*(G_0/B)$
from the symmetric algebra of $\widehat T$ to the Chow ring of $G_0/B$ given
in \cite[2.7]{PS10}, and denote
its image by $R^*$.
According to \cite[Rem.~$2^\circ$]{Gr58}, the ring $\CH^*(G_0)$ can be presented as
the quotient of $\CH^*(G_0/B)$ modulo the ideal generated by the
non-constant elements of $R^*$.

\begin{lem}\label{lem:comm} The pull-back map
$$
\CH^*(G_0)\to\CH^*([G_0,G_0])
$$
is an isomorphism.
\end{lem}
\begin{proof}
Indeed, $B'=B\cap [G_0,G_0]$ is a Borel subgroup of $[G_0,G_0]$, the map
$$
[G_0,G_0]/B'\to G_0/B
$$
is an isomorhism, and the map $S(\widehat T)\to\CH^*(G_0/B)$ factors through
the surjective map $S(\widehat T)\to S(\widehat T')$, where
$T'=T\cap [G_0,G_0]$.
\end{proof}

Let $P$ be a parabolic subgroup of $G_0$. Denote by $L$ the Levi subgroup
of $P$ and set $H_0=[L,L]$. We have

\begin{lem}\label{lem:levi}
The pull-back map
$$
\CH^*(P)\to\CH^*(H_0)
$$
is an isomorphism.
\end{lem}

\begin{proof}
The quotient map $P\to L$ is Zariski locally trivial affine fibration,
therefore the pull-back map $\CH^*(L)\to\CH^*(P)$ is an isomorphism.
Since the composition $L\to P\to L$ is the identity map, the pull-back
map $\CH^*(P)\to\CH^*(L)$ is an isomorphism as well. It remains to apply
Lemma~\ref{lem:comm}.
\end{proof}

\begin{lem}\label{lem:surj}
The pull-back map
$$
\CH^*(G_0)\to\CH^*(P)
$$
is surjective.
\end{lem}
\begin{proof}
Applying \cite[Proposition~3]{Gr58} to the natural map $G_0/B\to G_0/P$
we see that the map $\CH^*(G_0/B)\to\CH^*(P/B)$ is surjective. But the map
$\CH^*(P/B)\to\CH^*(P)$ is also surjective by Lemma~\ref{lem:levi} and fits
into the commutative diagram
$$
\xymatrix{
\CH^*(G_0/B)\ar@{>>}[r]\ar[d]&\CH^*(P/B)\ar@{>>}[d]\\
\CH^*(G_0)\ar[r]&\CH^*(P).
}
$$
\end{proof}

\begin{ntt}[Definition of $\sigma$]
Now we restrict to the situation when $G_0$ is simple. Let $p$ be a prime integer.
Denote $\Ch^*(-)$ the Chow ring with $\ff_p$-coefficients.
Explicit presentations of the Chow rings with $\ff_p$-coefficients of split semisimple algebraic groups
are given in \cite[Theorem~3.5]{Kc85}.

For $G_0$ and $H_0$ they look as follows:
\begin{align*}                                               
&\Ch^*(G_0)=\ff_p[x_1,\ldots,x_r]/(x_1^{p^{k_1}},\ldots,x_r^{p^{k_r}})
\text{ with }\deg x_i=d_i, 1\le d_1\le\ldots\le d_r;\\
&\Ch^*(H_0)=\ff_p[y_1,\ldots,y_s]/(y_1^{p^{l_1}},\ldots,
y_{s}^{p^{l_s}})\text{ with }\deg y_m=e_m, 1\le e_1\le\ldots\le e_s
\end{align*}
for some integers $k_i$, $l_i$, $d_i$, and $e_i$ depending on the
Dynkin types of $G_0$ and $H_0$.

By the previous lemmas the pull-back $\varphi\colon\Ch^*(G_0)\to\Ch^*(H_0)$
is surjective. For a graded ring $S^*$ denote by $S^+$ the ideal
generated by the non-constant elements of $S^*$.
The induced map
$$\Ch^+(G_0)/\Ch^+(G_0)^2\to\Ch^+(H_0)/\Ch^+(H_0)^2$$
is also surjective.
Moreover, for any $m$ with $e_m>1$ there exists a unique $i$ such that $d_i=e_m$.
We denote $i=:\sigma(m)$. The surjectivity implies that
$$
\varphi(x_{\sigma(m)})=cy_m+\text{lower terms},\quad c\in\ff_p^\times.
$$
\end{ntt}

\section{Generically split varieties}

For a semisimple group $G$ and a prime number $p$
denote by
$$J_p(G)=(j_1(G),\ldots,j_r(G))$$
its {\it $J$-invariant} defined in \cite{PSZ08}.

\begin{thm}\label{thm1}
Let $G_0$ be a split simple algebraic group over $k$,
$G={}_\gamma G_0$ be the twisted form of $G_0$ 
given by a $1$-cocycle $\gamma\in H^1(k,G_0)$,
$X={}_\gamma(G_0/P)$ be the twisted form of $G_0/P$,
and $Y={}_\gamma(G_0/B)$ be the twisted form of $G_0/B$.
The following conditions are equivalent:
\begin{enumerate}
\item $X$ is generically split;

\item The composition map
$$
\CHO^*(Y)\to\CH^*(G_0)\to\CH^*(P)
$$
is surjective;

\item For every prime $p$ the composition map
$$
\ChO^1(Y)\to\Ch^1(G_0)\to\Ch^1(P)
$$
is surjective, and
$$
j_{\sigma(m)}(G)=0\text{ for all }m\text{ with }d_m>1.
$$
\end{enumerate}
\end{thm}
\begin{proof}
1$\Rightarrow$2. The same argument as in the proof of Lemma~\ref{lem:surj}
(with $Y$ instead of $G_0/B$ and $X$ instead of $G_0/P$).

2$\Rightarrow$3. Clearly, the composition
$$
\ChO^*(Y)\to\Ch^*(G_0)\to\Ch^*(P)
$$
is surjective for every $p$. In particular, when $d_m>1$ $\ChO^{d_m}(Y)$
contains an element of the form $x_{\sigma(m)}+a$, where $a$ is decomposable,
hence $j_{\sigma(m)}(G)=0$.

3$\Rightarrow$1. $G_{k(X)}$ has a parabolic subgroup of type $P$; denote
the derived group of its Levi subgroup by $H$. We want to prove that $H$
is split. By \cite[Proposition~3.9(3)]{PS10} it suffices to show that
$J_p(H)$ is trivial for every $p$. 

Denote the variety of complete flags of $H$ by $Z$.
It follows from the commutative diagram
$$
\xymatrix{
\Ch^*(Y_{k(X)})\ar[r]\ar[d]&\Ch^*(Z)\ar[d]\\
\Ch^*(\BG)\ar[r]&\Ch^*(\overline{H})
}
$$
that $j_m(H)\le j_{\sigma(m)}(G)$ if $d_m>1$.
Therefore
$$
j_m(H)\le j_{\sigma(m)}(G_{k(X)})\le j_{\sigma(m)}(G)=0
$$
when $d_m>1$. It remains to show that $\Ch^1(\overline{Z})$
is rational.
But this follows from the commutative diagram
$$
\xymatrix{
\Ch^1(Y)\ar[r]&\Ch^1(Y_{k(X)})\ar[r]\ar[d]&\Ch^1(Z)\ar[d]&\\
&\Ch^1(\BG)\ar[r]&\Ch^1(\overline{H})\ar@{=}[r]&\Ch^1(P).
}
$$
\end{proof}

\begin{rem}\mbox{}
\begin{itemize}
\item If all $e_m>1$, then the condition on $\ChO^1(Y)$ is void.
\item If $G_0$ is different from $\PGO_{2n}^+$ and $e_1=1$
(resp. $G_0=\PGO_{2n}^+$ and $e_1=e_2=1$), then
in view of \cite[Proposition~4.2]{PS10} it is equivalent
to the fact that all Tits algebras of $G$ are split. The latter is
also equivalent to the fact that $j_1(G)=0$ (resp. $j_1(G)=j_2(G)=0$).
\item If $G_0=\PGO_{2n}^+$ and there is exactly one $m$ with $e_m=1$, then
there are exactly two fundamental weights among $\w_1,\w_{n-1},\w_n$
whose image with respect to the composition
$\Ch^1(\overline Y)\to\Ch^1(\overline G)\to\Ch^1(\overline H)$
equals $y_1$.
Then the condition on $\ChO^1(Y)$ is equivalent to the fact that
at least one of the Tits algebras corresponding to these fundamental weights
in the preimage of $y_1$ is split.
\end{itemize}
\end{rem}

For a simple group $G$ we denote by $A_l$ its Tits algebra corresponding to $\w_l$.

\begin{thm}\label{cor1}
Let $G$ be a group given by a $1$-cocycle from $H^1(k,G_0)$, where $G_0$
stands for the split adjoint group of the same type as $G$,
and let $X$ be the variety of the parabolic subgroups of $G$ of type $i$.

The variety $X$ is generically split if and only if

\medskip

\begin{longtable}{l|l|l}
$G_0$&$i$&conditions on $G$\\\hline
$\PGL_n$&any $i$ & $\gcd(\mathrm{exp} A_1,i)=1$\\
$\PGSp_{2n}$& any $i$ &$i$ is odd or $G$ is split\\
$\mathrm{O}^+_{2n+1}$&any $i$&$j_m(G)=0$ for all $1\le m\le\big[\frac{n+1-i}{2}\big]$\\\hline
$\PGO^+_{2n}$&$i$ is odd, $i<n-1$&$[A_{n-1}]=0$ or $[A_n]=0$, and\\
&&$j_m(G)=0$ for all $2\le m\le\big[\frac{n+2-i}{2}\big]$\\\hline
$\PGO^+_{2n}$&$i$ is even, $i<n-1$&$j_m(G)=0$ for all $1\le m\le\big[\frac{n+2-i}{2}\big]$\\
$\PGO^+_{2n}$&$i=n-1$ or $i=n$, $n$ is odd&none\\
$\PGO^+_{2n}$&$i=n-1$, $n$ is even&$[A_1]=0$ or $[A_n]=0$\\
$\PGO^+_{2n}$&$i=n$, $n$ is even&$[A_1]=0$ or $[A_{n-1}]=0$\\\hline
$\E_6$&$i=3,5$&none\\
$\E_6$&$i=2,4$&$J_3(G)=(0,*)$\\
$\E_6$&$i=1,6$&$J_2(G)=(0)$\\
$\E_7$&$i=2,5$&none\\
$\E_7$&$i=3,4$&$J_2(G)=(0,*,*,*)$\\\hline
$\E_7$&$i=6$&$J_2(G)=(0,0,*,*)$\\
&&($J_2(G)=(0,0,0,0)$ if $\charact k\ne 2$)\\\hline
$\E_7$&$i=1$&$J_2(G)=(0,0,0,*)$\\
&&($J_2(G)=(0,0,0,0)$ if $\charact k\ne 2$)\\\hline
$\E_7$&$i=7$&$J_3(G)=(0)$ and $J_2(G)=(*,0,*,*)$\\
&&($J_2(G)=(*,0,0,0)$ if $\charact k\ne 2$)\\\hline
$\E_8$&$i=2,3,4,5$&none\\\hline
$\E_8$&$i=6$&$J_2(G)=(0,*,*,*)$\\
&&($J_2(G)=(0,0,0,*)$ if $\charact k\ne 2$)\\\hline
$\E_8$&$i=1$&$J_2(G)=(0,0,*,*)$\\
&&($J_2(G)=(0,0,0,*)$ if $\charact k\ne 2$)\\\hline
$\E_8$&$i=7$&$J_3(G)=(0,*)$ and $J_2(G)=(0,*,*,*)$\\
&&($J_3(G)=(0,0)$ if $\charact k\ne 3$,\\
&&$J_2(G)=(0,0,0,*)$ if $\charact k\ne 2$)\\\hline
$\E_8$&$i=8$&$J_3(G)=(0,*)$ and $J_2(G)=(0,0,0,*)$\\
&&($J_3(G)=(0,0)$ if $\charact k\ne 3$)\\\hline
$\F$&$i=1,2,3$&none\\
$\F$&$i=4$&$J_2(G)=(0)$\\
$\G$&any $i$&none
\end{longtable}

{\footnotesize (``$*$'' means ``any value'').}
\end{thm}
\begin{proof}
Follows immediately from Theorem~\ref{thm1} and \cite[Table~4.13]{PSZ08}.
\end{proof}

This theorem allows to give a shortened proof of the main result of
\cite{Ch10}:

\begin{cor}
Let $G$ be a group of type $\E_8$ over a field $k$ with $\charact k\ne 3$.
If the $3$-component of the Rost invariant of $G$ is zero, then $G$ splits over a field
extension of degree coprime to $3$.
\end{cor}
\begin{proof}
Let $K/k$ be a field extension of degree coprime to $3$ such that
the $2$-component of the Rost invariant of $G_K$ is zero.

Consider the variety $X$ of parabolic subgroups of $G_K$ of type $7$.
The Rost invariant of the semisimple anisotropic kernel of $G_{K(X)}$ is zero.
Therefore $G_{K(X)}$ splits, and, thus, $X$ is generically split.

By Theorem~\ref{cor1} $J_3(G_K)=(0,0)$, hence by \cite[Proposition~3.9(3)]{PS10}
$G_K$ splits over a field extension of degree coprime to $3$. This implies
the corollary.
\end{proof}

\bibliographystyle{chicago}

\noindent
{\sc
V.~PETROV\\
Max-Planck-Institut f\"ur Mathematik, D-53111 Bonn, Germany}

\noindent
{\tt E-mail: victorapetrov@googlemail.com}

\medskip

\noindent
{\sc N.~SEMENOV\\
Johannes Gutenberg-Universit\"at Mainz, Institut f\"ur Mathematik, Staudingerweg 9,
D-55099 Mainz, Germany}

\noindent
{\tt E-mail: semenov@uni-mainz.de}
\end{document}